\documentclass[12pt,a4paper]{amsart}
\usepackage{amsmath}
\usepackage{amsthm}
\usepackage{amssymb}
\usepackage{drpack}
\usepackage[english]{babel}
\usepackage{verbatim}
\usepackage[shortlabels]{enumitem}
\usepackage{color}
\usepackage{tikz-cd}

\newtheoremstyle{mio}%
{}{} 
{\itshape}{} 
{\bfseries}{.}{ } 
{#1 #2\thmnote{~\mdseries(#3)}} 
\theoremstyle{mio}
\newtheorem{teor}{Theorem}[section]
\newtheorem{cor}[teor]{Corollary}
\newtheorem{prop}[teor]{Proposition}
\newtheorem{lemma}[teor]{Lemma}
\newtheorem{defin}[teor]{Definition}

\newtheoremstyle{definition2}%
{}{} 
{}{} 
{\bfseries}{.}{ } 
{#1 #2\thmnote{\mdseries~ #3}} 
\theoremstyle{definition2}
\newtheorem{ex}[teor]{Example}
\newtheorem{oss}[teor]{Remark}

\title{The local Picard group of a ring extension}
\author{Dario Spirito}
\address{Dipartimento di Scienze Matematiche, Fisiche e Informatiche, Universit\`a di Udine, Udine, Italy}
\email{dario.spirito@uniud.it}

\newcommand{\Int}{\mathrm{Int}}
\newcommand{\Pic}{\mathrm{Pic}}

\newcommand{\njaff}{\mathcal{N}}
\newcommand{\locpic}{\mathrm{LPic}}
\newcommand{\Inv}{\mathrm{Inv}}
\newcommand{\Princ}{\mathrm{Princ}}
\newcommand{\XX}{\mathbf{X}}
\newcommand{\B}{\mathcal{B}}

\begin{document}
\begin{abstract}
Given an integral domain $D$ and a $D$-algebra $R$, we introduce the local Picard group $\locpic(R,D)$ as the quotient between the Picard group $\Pic(R)$ and the canonical image of $\Pic(D)$ in $\Pic(R)$, and its subgroup $\locpic_u(R,D)$ generated by the the integral ideals of $R$ that are unitary with respect to $D$. We show that, when $D\subseteq R$ is a ring extension that satisfies certain properties (for example, when $R$ is the ring of polynomial $D[X]$ or the ring of integer-valued polynomials $\Int(D)$), it is possible to decompose $\locpic(R,D)$ as the direct sum $\bigoplus\locpic(RT,T)$, where $T$ ranges in a Jaffard family of $D$. We also study under what hypothesis this isomorphism holds for pre-Jaffard families of $D$.
\end{abstract}

\maketitle

\section{Introduction}
Let $D$ be an integral domain. The \emph{Picard group} of $D$, denoted by $\Pic(D)$, is the quotient between the group of invertible (fractional) ideals and the group of principal ideals or, equivalently, the group of isomorphism classes of rank-one projective modules. The Picard group of $D$ is connected, among other topics, to the factorization properties of $D$: for example, if $D$ is a Dedekind domain, then $\Pic(D)$ is trivial if and only if $D$ is a unique factorization domain.

The Picard group is essentially a global property of a ring, in the sense  that it cannot be recovered from the localizations: indeed, the Picard group of a local ring is always trivial. However, in some cases it is possible to understand the structure of the Picard group by choosing carefully a family of localizations: for example, when considering the ring $R=\Int(D)$ of the integer-valued polynomials on a Dedekind domain $D$, there is an exact sequence
\begin{equation}\label{eq:PicInt-intro}
0\longrightarrow\Pic(D)\longrightarrow\Pic(\Int(D))\longrightarrow\bigoplus_{M\in\Max(D)}\Pic(\Int(D_M))\longrightarrow 0,
\end{equation}
where $\Int(D_M)=\Int(D)D_M$ is a localization of $\Int(D)$ \cite[Theorem VIII.I.9]{intD}. Since each $D_M$ is a valuation domain, $\Pic(\Int(D_M))$ is known \cite[Theorem VIII.2.8]{intD}, and thus we can calculate $\Pic(D)$ by localization.

The previous result was generalized in \cite{PicInt} in the context of Jaffard families: a \emph{Jaffard family} of $D$ is a family of flat overrings of $D$ that is complete, independent and locally finite (see Section \ref{sect:prelim:jaffard} for the definitions of these properties). A typical example of a Jaffard family is the family $\{D_M\mid M\in\Max(D)\}$ of localizations of a one-dimensional locally finite domain (for example, a one-dimensional Noetherian domain). It was shown \cite[Proposition 4.3]{PicInt} that, if $\Theta$ is a Jaffard family of $D$, then we can construct an exact sequence analogous to \eqref{eq:PicInt-intro} for $\Pic(\Int(D))$, and subsequently \cite[Theorem 4.7]{PicInt} that there is an isomorphism
\begin{equation}\label{eq:locpic-intro}
\locpic(\Int(D),D)\simeq\bigoplus_{T\in\Theta}\locpic(\Int(T),T),
\end{equation}
where $\locpic(\Int(A),A)$ is the quotient between $\Pic(\Int(A))$ and the canonical image of $\Pic(A)$. (We use the notation we introduce in Section \ref{sect:picard} of the present paper instead of the notation used in \cite{PicInt}.) These results were then extended beyond the Jaffard family context with a derived set-like construction \cite[Sections 6 and 7]{PicInt} first introduced in \cite{jaff-derived}.

The purpose of this paper is to extend \eqref{eq:PicInt-intro} and \eqref{eq:locpic-intro} beyond the case of integer-valued polynomials to the more general framework of $D$-algebras. In particular, we want to understand when, given a Jaffard family $\Theta$ of $D$ and a $D$-algebra $R$, we have an isomorphism
\begin{equation}\label{eq:locpic-intro-2}
\locpic(R,D)\simeq\bigoplus_{T\in\Theta}\locpic(RT,T),
\end{equation}
fully reducing the study of the local Picard group $\locpic(R,D)$ to localizations; moreover, we want to see whether it is possible to generalize these results to the more general case of pre-Jaffard families. However, the isomorphism \eqref{eq:locpic-intro} does not hold for general $D$-algebras: thus, the first part of the paper (Sections \ref{sect:retract}, \ref{sect:unitary}, \ref{sect:pseudopol}) is devoted to introducing and exploring the additional hypothesis we need to put on $R$ in order for the isomorphism to hold, while the second part (from Section \ref{sect:picard} onward) adapts the proof of the integer-valued polynomial case to this more general setup. We shall see that several polynomial-like constructions satisfy these conditions: for example, \eqref{eq:locpic-intro-2} holds not only for $R=\Int(D)$ (which is the topic of \cite{PicInt}), but also for $R=\Int(E,D)$ (the ring of integer-valued polynomials on any $E\subseteq D$), the polynomial ring $R=D[X]$ and the Bhargava ring $R=\mathbb{B}_x(D)$ (see \cite{yeramian-bhargava}).

More specifically, in Section \ref{sect:retract} we study $D$-algebras $R$ that can be endowed with a retract, i.e., with a map $R\longrightarrow D$ that is a $D$-algebra homomorphism, showing that they are all extensions of $D$ such that $R\cap K=D$ (where $K$ is the quotient field of $D$; Proposition \ref{prop:retract}). In Section \ref{sect:unitary}, we develop the theory of unitary ideals: a fractional ideal $I$ of $R$ is unitary with respect to $D$ if $I\cap K\neq(0)$. We introduce the subgroup $\Pic_u(R,D)$ of $\Pic(R)$ as the subgroup generated by the classes of the unitary integral ideals of $R$, and we show that $\Pic(R)/\Pic_u(R,D)$ is isomorphic to $\Pic(RK)$ (Proposition \ref{prop:quoz-picu}). In Section \ref{sect:pseudopol}, we introduce pseudo-polynomial algebras over $D$ as those algebras where every unitary principal ideal is actually generated by an element of $D$, and we show (Proposition \ref{prop:pseudopol-pol}) that this notion encompasses several constructions contained in the ring of polynomials $K[X]$.

In Section \ref{sect:picard} we introduce the local Picard group $\locpic(R,D)$ and the unitary local Picard group $\locpic_u(R,D)$ as, respectively, the quotient of $\Pic(R)$ and $\Pic_u(R,D)$ by the canonical image of $\Pic(D)$. We show that these constructions are functorial (Proposition \ref{prop:locpic-funct}) and that for retract $D$-algebras they are actually direct summand of $\Pic(R)$ and $\Pic_u(R,D)$, respectively (Proposition \ref{prop:locpic-directsum}).

In Section \ref{sect:jaffard} we take the proofs of \eqref{eq:PicInt-intro} and \eqref{eq:locpic-intro} and show how they can be adapted to the case of $D$-algebras: in order to hold in the more general context, we need to restrict ourselves to algebras that are retract and pseudo-polynomial and, instead to $\locpic(R,D)$, the best results are obtained when dealing with the group $\locpic_u(R,D)$ induced by unitary integral ideals (Theorems \ref{teor:exseq-inttype} and \ref{teor:exseq-inttype-jaffard}). We also show several special cases of these theorems. Finally, in Section \ref{sect:preJaff}, we show under what hypothesis the results about Jaffard families can be generalized to pre-Jaffard families using the derived sequence.

\section{Preliminaries}\label{sect:prelim}
Throughout the paper, $D$ is an integral domain and $K$ is its quotient field. We also suppose that $D\neq K$, i.e., that $D$ is not a field.

A \emph{fractional ideal} of $D$ is a $D$-submodule $I$ of $K$ such that $dI\subseteq D$ for some $d\in D$, $d\neq 0$. We shall often refer to a fractional ideal simply as an ``ideal'', while using the term ``integral ideal'' to refer to fractional ideals contained in $D$ (i.e., to ideals of $D$ in the usual sense).

An \emph{overring} of $D$ is a ring between $D$ and $K$. If $D\subseteq R$ is a ring extension, with $R$ being an integral domain, then the quotient field of $R$ contains $K$, and thus it makes sense to consider the intersection $R\cap K$. Moreover, if $T$ is an overring of $D$, then the set
\begin{equation*}
RT:=\left\{\sum_ir_it_i\mid r_i\in R,t_i\in T\right\}
\end{equation*}
is well-defined, and it is a ring and an overring of $R$. In particular, if $T=S^{-1}D$ is a localization of $D$, then $RT=S^{-1}R$ is a localization of $R$.

We say that two elements $a,b\in D$ are \emph{associated} in $D$ if they generate the same principal ideal, or equivalently if there is a unit $u$ such that $a=ub$.

If $E$ is a subset of $K$, the ring of \emph{integer-valued polynomials} on $E$ is
\begin{equation*}
\Int(E,D):=\{f\in K[X]\mid f(E)\subseteq D\};
\end{equation*}
we also set $\Int(D):=\Int(D,D)$. See \cite{intD} for facts about integer-valued polynomials.

\subsection{The Picard group}\label{sect:prelim:pic}
A fractional ideal $I$ of $D$ is \emph{invertible} if there is a fractional ideal $J$ such that $IJ=D$; in this case, $J=(D:I)=\{x\in K\mid xI\subseteq D\}$. Every invertible ideal is finitely generated. The set $\Inv(D)$ of all the invertible ideals of $D$ is a group under the product of ideals, with identity element $D$, and contains as a subgroup the set $\Princ(D)$ of all the principal fractional ideals of $D$. The quotient
\begin{equation*}
\Pic(D):=\frac{\Inv(D)}{\Princ(D)}
\end{equation*}
is called the \emph{Picard group} of $D$; if $I$ is an invertible ideal, we denote by $[I]$ the class of $I$ in $\Pic(D)$. The Picard group can also be constructed as the set of all isomorphism classes of all projective modules of rank $1$, with operation given by $[P]\cdot[Q]:=[P\otimes Q]$.

The Picard group is a functorial construction, in the sense that a ring homomorphism $\phi:A\longrightarrow B$ induces a map $\phi^\ast:\Pic(A)\longrightarrow\Pic(B)$. If $A,B$ are domains and $\phi$ is injective then $\phi^\ast$ coincides with the extension of ideals: that is, if $I$ is an invertible ideal of $A$, then $\phi^\ast([I])=[IB]$.

\subsection{Jaffard and pre-Jaffard families}\label{sect:prelim:jaffard}
A \emph{flat overring} of $D$ is an overring that is also flat as a $D$-module. Let $\Theta$ be a family of flat overrings of $D$: we say that $\Theta$ is:
\begin{itemize}
\item \emph{complete} if $I=\bigcap\{IT\mid T\in\Theta\}$ for every ideal $I$ of $D$;
\item \emph{independent} if $TT'=K$ for every $T\neq T'$ in $K$;\footnote{For general (not necessarily flat) overrings, independence requires that there is no nonzero prime $P$ such that $PT\neq T$ and $PT'\neq T'$ \cite[Section 6.2]{fontana_factoring}, but this condition reduces to $TT'=K$ for flat overrings \cite[Lemma 3.4 and Definition 3.5]{jaff-derived}.}
\item \emph{locally finite} if for every $x\in D$, $x\neq 0$, there are only finitely many $T\in\Theta$ such that $x$ is not a unit in $D$.
\end{itemize}

A \emph{Jaffard family} of $D$ is a family of flat overrings that is complete, independent and locally finite, and such that $K\notin\Theta$. In particular, if $\Theta$ is a Jaffard family, for every nonzero prime ideal $P$ of $D$ there is a unique $T\in\Theta$ such that $PT\neq T$. If $T$ is a flat overring of $D$, then $T$ is a \emph{Jaffard overring} if $T$ belongs to a Jaffard family of $D$; this condition can be checked by defining the \emph{orthogonal} of $T$ (with respect to $D$) as
\begin{equation*}
T^\perp:=\bigcap\{D_P\mid PT=T\}.
\end{equation*}
Indeed, $T$ is a Jaffard overring if and only if $TT^\perp=K$, and in this case $\{T,T^\perp\}$ is a Jaffard family. See \cite{starloc} and \cite[Section 6.3]{fontana_factoring} for properties of Jaffard families.

The \emph{Zariski topology} on the set $\Over(D)$ of overrings of $D$ is the topology generated by the $\B(x):=\{T\in\Over(D)\mid x\in T\}$, as $x$ ranges in $K$. A \emph{pre-Jaffard family} of $D$ is a family $\Theta$ of flat overrings that is complete, independent, such that $K\notin\Theta$, and compact in the Zariski topology. Any Jaffard family is also a pre-Jaffard family, and a pre-Jaffard family $\Theta$ is Jaffard if and only if every $T\in\Theta$ is a Jaffard overring. If $\Theta$ is a pre-Jaffard family and there is a $T_\infty\in\Theta$ such that every $T\in\Theta\setminus\{T_\infty\}$ is a Jaffard overring, we say that $\Theta$ is a \emph{weak Jaffard family pointed at $T_\infty$}.

Let  $\Theta$ be a pre-Jaffard family. We can associate to $\Theta$ two sequences, one $\{\njaff^\alpha(\Theta)\}_\alpha$ of subsets of $\Theta$, and another $\{T_\alpha\}_\alpha$ of overrings of $D$, both indexed by ordinal numbers, in the following way:
\begin{itemize}
\item $\njaff^0(\Theta):=\Theta$, $T_0:=D$;
\item if $\alpha=\gamma+1$ is a successor ordinal, then $\njaff^\alpha(\Theta)$ is the set of members of $\njaff^\gamma(\Theta)$ that are not Jaffard overrings of $T_\gamma$;
\item if $\alpha$ is a limit ordinal, then $\njaff^\alpha(\Theta):=\bigcap\{\njaff^\gamma(\Theta)\mid \gamma<\alpha\}$;
\item $T_\alpha:=\bigcap\{T\mid T\in\njaff^\alpha(\Theta)\}$.
\end{itemize}
Then, $\{\njaff^\alpha(\Theta)\}_\alpha$ is decreasing and $\{T_\alpha\}_\alpha$ is increasing; moreover, $\njaff^\alpha(\Theta)$ is always a pre-Jaffard family of $T_\alpha$. A weak Jaffard family is just a pre-Jaffard family such that $\njaff^1(\Theta)$ is empty or a singleton. We call $\{T_\alpha\}_\alpha$ the \emph{derived sequence} with respect to $\Theta$. If $T_\alpha=K$ for some $\alpha$, we say that the pre-Jaffard family $\Theta$ is \emph{sharp}. See \cite{jaff-derived} for properties of pre-Jaffard families and of the derived sequence.

\section{Retract algebras}\label{sect:retract}
\begin{defin}
Let $D$ be an integral domain, and let $R$ be a $D$-algebra that is an integral domain. We say that $R$ is a \emph{retract $D$-algebra} if there is a $D$-algebra homomorphism $\epsilon:R\longrightarrow D$, which we call a \emph{retract} of $R$ onto $D$.
\end{defin}

\begin{ex}
~\\
\begin{enumerate}
\item Let $R=D[X]$. For any $i\in D$, the evaluation homomorphism $\epsilon_i(f(X))=f(i)$ is a $D$-algebra homomorphism, and thus $D[X]$ is a retract $D$-algebra. The same holds for polynomial rings $D[X_1,\ldots,]$ in any number of indeterminates.
\item If $K$ is the quotient field of $D$ and $i\in K$, then for every $E\ni 1$ the evaluation homomorphisms $\epsilon_i(f(X))=f(i)$ also makes $\Int(E,D)$ into a retract $D$-algebra. In particular, taking $i\in D$, we see that $\Int(D)$ is a retract $D$-algebra.
\item The power series ring $D[[X]]$ is a retract $D$-algebra, with retract given by the evaluation at $0$, i.e., with the map that associates to the power series $f_0+f_1X+\cdots$ the constant term $f_0$.
\end{enumerate}
\end{ex}

\begin{lemma}\label{lemma:retract-inj}
If $\epsilon:R\longrightarrow D$ is a $D$-algebra homomorphism and $R\neq D$, then $\epsilon$ is not injective.
\end{lemma}
\begin{proof}
If $r\in R\setminus D$ then $r\neq\epsilon(r)$ and $\epsilon(r)=\epsilon(\epsilon(r))$ since $\epsilon(r)\in D$, and thus $\epsilon$ is not injective.
\end{proof}

\begin{prop}\label{prop:retract}
Let $D$ be an integral domain, and let $R$ be a retract $D$-algebra.
\begin{enumerate}[(a)]
\item\label{prop:retract:ext} $R$ is an extension of $D$ (i.e., the canonical $D$-algebra homomorphism $D\longrightarrow R$ is injective).
\item\label{prop:retract:subring} If $S\subseteq R$ is a $D$-algebra, then $S$ is a retract $D$-algebra.
\item\label{prop:retract:capK} If $K$ is the quotient field of $D$, then $R\cap K=D$.
\item\label{prop:retract:Spec} The restriction map of spectra
\begin{equation*}
\begin{aligned}
\Spec(R) & \longrightarrow \Spec(D),\\
P& \longmapsto P\cap D,
\end{aligned}
\end{equation*}
is surjective.
\item\label{prop:retract:dim} If $R\neq D$, then $\dim(R)\geq\dim(D)+1$.
\end{enumerate}
\end{prop}
\begin{proof}
Let $\epsilon$ be a retract of $R$ onto $D$.

\ref{prop:retract:ext} Let $i:D\longrightarrow R$ be the canonical map of $D$-algebras. Then, $\epsilon\circ i$ is the identity on $D$; therefore, $i$ must be injective, i.e., $R$ is an extension of $D$.

\ref{prop:retract:subring} The restriction of $\epsilon$ to $S$ is a $D$-algebra homomorphism. Hence $S$ is a retract $D$-algebra.

\ref{prop:retract:capK} Let $D\subseteq D'\subseteq K$. Then, any ring homomorphism $\phi:D'\longrightarrow A$ is uniquely determined by the restriction $\phi|_D:D\longrightarrow A$, since any $x\in D'$ can be written as a quotient $y/z$ for $y,z\in D$. (That is, the extension $D\subseteq D'$ is an epimorphism.) In particular, if $R$ is a retract $D$-algebra and $D':=R\cap D$, then the retract $\epsilon$ is a ring homomorphism that is the identity on $D$; hence, it must be the identity on $D'$. However, this is impossible if $D'\neq D$. Thus $D'=R\cap K=D$.

\ref{prop:retract:Spec} The map $\epsilon$ induces a map $\epsilon^\ast:\Spec(D)\longrightarrow\Spec(R)$ given by $\epsilon^\ast(P)=\epsilon^{-1}(P)$: we claim that $P=\epsilon^{-1}(P)\cap D$. Indeed, since $\epsilon$ is the identity on $D$, we have $P\subseteq\epsilon^{-1}(P)\cap D$; likewise, if $d\in\epsilon^{-1}(P)\cap D$, then $d=\epsilon(d)\in P$ and so $\epsilon^{-1}(P)\cap D\subseteq P$. Hence the restriction map is surjective.

\ref{prop:retract:dim} If $\dim(D)$ is infinite the claim is trivial. Suppose that $\dim(D)<\infty$: then, since $\epsilon$ is surjective, $\Spec(D)$ is homeomorphic to $V(\ker\epsilon)\subseteq\Spec(R)$. Since $D$ and $R$ are integral domains, the dimensions of $\Spec(R)$ and $\Spec(D)$ can coincide only if $\ker\epsilon=(0)$, i.e., if $\epsilon$ is injective. This is impossible by Lemma \ref{lemma:retract-inj}. Therefore $\dim(R)\geq\dim(D)+1$, as claimed.
\end{proof}

\begin{oss}
If $R=D[X]$ or $R=\Int(D)$, there is more than one possible map $\epsilon$. In this case, the prime $\epsilon^{-1}(P)$ may change when changing the map, but its image under the restriction map is always $P$. For example, if $\epsilon_d$ is the map of $D[X]$ defined as the evaluation in $d$, then $\epsilon_d^{-1}(P)=(P,X-d)$.
\end{oss}

\begin{prop}\label{prop:retract-integral}
Let $D\subsetneq R$ be an integral extension of domains. Then, $R$ is not a retract $D$-algebra.
\end{prop}
\begin{proof}
Suppose there is a retract $\epsilon:R\longrightarrow D$; since $R\neq D$, $\epsilon$ is not injective (Lemma \ref{lemma:retract-inj}) and thus $\ker\epsilon\neq(0)$. Let $S:=D\setminus\{0\}$: then, $S^{-1}D\subseteq S^{-1}R$ is again an integral extension, and $S^{-1}D=K$ is the quotient field of $R$. Hence, $S^{-1}R$ must be a field too; however, $\ker\epsilon\cap D=(0)$, and thus $S^{-1}\ker\epsilon\neq S^{-1}R$ is a nonzero prime ideal above $(0)$, a contradiction. Hence $R$ is not a retract $D$-algebra.
\end{proof}

\begin{oss}
The condition $R\cap K=D$ is not sufficient for $R$ to be a retract $D$-algebra. For example, if $D$ is integrally closed and $R\neq D$ is an integral extension of $D$, then $R\cap K=D$, but $R$ is not a retract $D$-algebra by Proposition \ref{prop:retract-integral}.

The condition is not sufficient even if $D$ is integrally closed in $R$: for example, let $D=\insZ$ and let $R=\insZ[\alpha,\beta]\subseteq\overline{\insQ(X)}$, where $X$ is an indeterminate over $\insZ$ and $\alpha,\beta$ satisfy
\begin{equation*}
\begin{cases}
\alpha^2=X,\\
\beta^2=3-X.
\end{cases}
\end{equation*}
Then, $\insZ$ is integrally closed in $R$ and $R\cap\insQ=\insZ$. Suppose there is a retract $\epsilon:R\longrightarrow D$: since $\alpha^2+\beta^2=3$, we also must have
\begin{equation*}
\epsilon(\alpha)^2+\epsilon(\beta)^2=\epsilon(3)=3,
\end{equation*}
which is impossible in $\insZ$.
\end{oss}

We shall often deal with extending a retract $D$-algebra by a flat overring. In this context, the following result is useful.
\begin{prop}\label{prop:retract-extension}
Let $D$ be an integral domain and let $T$ be an overring of $D$. If $R$ is a retract $D$-algebra, then $RT$ is a retract $T$-algebra.
\end{prop}
\begin{proof}
Let $\epsilon:R\longrightarrow D$ be a retract. If $x\in RT$, then $x=r_1t_1+\cdots+r_nt_n$ for some $r_i\in R$, $t_i\in T$; we define
\begin{equation*}
\widetilde{\epsilon}(x):=\epsilon(r_1)t_1+\cdots+\epsilon(r_n)t_n\in DT=T.
\end{equation*}

We need to show that $\widetilde{\epsilon}$ is well-defined, i.e., that if 
\begin{equation*}
x=r_1t_1+\cdots+r_nt_n=s_1t'_1+\cdots+s_mt'_m
\end{equation*}
for some $r_i,s_i\in R$, $t_i,t'_i\in T$, then
\begin{equation*}
\epsilon(r_1)t_1+\cdots+\epsilon(r_n)t_n=\epsilon(s_1)t'_1+\cdots+\epsilon(s_n)t'_n.
\end{equation*}
To do so, it is enough to show that if $r_1t_1+\cdots+r_nt_n=0$ then $\epsilon(r_1)t_1+\cdots+\epsilon(r_n)t_n=0$.

For each $i$ we can write $t_i=y_i/z_i$ with $y_i,z_i\in D$, $z_i\neq 0$; let $z:=z_1\cdots z_n$. Then, $zt_i\in R$ for every $i$ and $zx=0$. Thus,
\begin{equation*}
0=\epsilon(zx)=\epsilon(z(r_1t_1+\cdots+r_nt_n))=\epsilon(r_1)(zt_1)+\cdots+\epsilon(r_n)(zt_n)
\end{equation*}
since each $zt_i$ is an element of $D$. The equality
\begin{equation*}
\epsilon(r_1)(zt_1)+\cdots+\epsilon(r_n)(zt_n)=0
\end{equation*}
is an equality in $K$ (the quotient field of $D$); therefore, we can simplify $z$ and obtain $\epsilon(r_1)t_1+\cdots+\epsilon(r_n)t_n=0$, which is what we needed to prove.

By construction, $\widetilde{\epsilon}$ is an homomorphism of $T$-algebras, and thus it is a retract. Hence $RT$ is a retract $T$-algebra.
\end{proof}

\begin{prop}
Let $D\subseteq R\subseteq A$ be integral domains. If $R$ is a retract $D$-algebra and $A$ is a retract $R$-algebra, then $A$ is a retract $D$-algebra.
\end{prop}
\begin{proof}
If $\epsilon:R\longrightarrow D$ and $\epsilon':A\longrightarrow R$ are retracts, then $\epsilon\circ\epsilon':A\longrightarrow D$ is a retract.
\end{proof}

\section{Unitary ideals}\label{sect:unitary}
An important feature of the theory of integer-valued polynomials is the concept of unitary ideals, i.e., ideals of $\Int(D)$ that meet $D$ nontrivially. We generalize this definition to any extension in the following way.
\begin{defin}
Let $D$ be an integral domain and $R$ an extension of $D$ that is an integral domain. We say that a fractional ideal $I$ of $R$ is \emph{unitary with respect to $D$} if $I\cap K\neq(0)$.
\end{defin}

When dealing with integral ideals, the defining condition of a unitary ideal can be written in a slightly different way.
\begin{lemma}\label{lemma:unitary-integral}
Let $D$ be an integral domain with quotient field $K$, and let $R$ be an extension of $D$. Let $I$ be an integral ideal of $R$. Then, $I$ is unitary if and only if $I\cap D\neq(0)$.
\end{lemma}
\begin{proof}
If $I\cap D\neq(0)$, then also $I\cap K\neq(0)$. Conversely, if $x\in I\cap K$, $x\neq 0$, then there is a $y\in D$, $y\neq 0$ such that $yx\in D$. Thus $yx\in yI\cap D\subseteq I\cap D$ and $I\cap D\neq(0)$.
\end{proof}

\begin{prop}
Let $D$ be an integral domain and $R$ an extension of $D$  that is an integral domain. An integral ideal $I$ of $R$ is unitary if and only if $IK=RK$.
\end{prop}
\begin{proof}
The ring $RK$ is the localization of $R$ at $S:=D\setminus\{0\}$; the claim now follows from Lemma \ref{lemma:unitary-integral}.
\end{proof}

In this paper, we are mainly interested in the study of the Picard group of a ring. In general, whenever $R$ is a $D$-algebra not contained in $K$, there will be plenty of ideals of $R$ that are not unitary; however, every ideal can be transformed by multiplication into an unitary ideal, in the sense that, if $x\in I$, $x\neq 0$ then $x^{-1}I$ is unitary (it contains $1$). In particular, if we denote by $\Inv_0(R,D)$ the set of all invertible ideals that are unitary with respect to $D$ (which is a subgroup of the group $\Inv(R)$ of all invertible ideals) then the quotient
\begin{equation*}
\frac{\Inv_0(R,D)}{\Inv_0(R,D)\cap\Princ(R)}=\frac{\Inv_0(R,D)}{\Princ_u(R,D)}
\end{equation*}
(where $\Princ(R)$ and $\Princ_u(R,D)$ are, respectively, the subgroup of principal ideals and of unitary principal ideals of $R$) is just equal to the Picard group $\Pic(R)$, since for every $I$ the coset $I\cdot\Princ(R)$ contains elements of $\Inv_0(R,D)$.

A more interesting way to consider unitary ideals is the following.
\begin{defin}
Let $R$ be a $D$-algebra. We define $\Inv_u(R,D)$ as the subgroup of $\Inv(R)$ generated by the ideals that are both integral and unitary with respect to $D$.

Furthermore, we define the \emph{unitary Picard group} of $R$ with respect to $D$ as the quotient
\begin{equation*}
\Pic_u(R,D):=\frac{\Inv_u(R,D)}{\Inv_u(R,D)\cap\Princ(R)}=\frac{\Inv_u(R,D)}{\Princ_u(R,D)}.
\end{equation*}
\end{defin}

\begin{prop}\label{prop:quoz-picu}
Let $R$ be a $D$-algebra that is an integral domain. Then, $\displaystyle{\frac{\Pic(R)}{\Pic_u(R,D)}\simeq\Pic(RK)}$.
\end{prop}
\begin{proof}
Let $\phi:\Pic(R)\longrightarrow\Pic(RK)$, $[I]\mapsto[IK]$ be the canonical map of Picard groups. We claim that its kernel is exactly $\Pic_u(R,D)$.

Since the set of all integral unitary invertible ideals is a monoid, every $[I]\in\Pic_u(R,D)$ can be written as $JL^{-1}$, where $J,L$ are integral unitary ideals of $R$; therefore, to show that $\Pic_u(R,D)\subseteq\ker\phi$, it is enough to prove it for integral unitary ideals. Let thus $[I]$ be integral and unitary: then, $IK\subseteq RK$, while $K\subseteq IK$ since $I\cap K\neq(0)$. Therefore, $IK=RK$ and in particular $IK$ is principal in $RK$, i.e., $[IK]=[RK]$ and $[I]\in\ker\phi$, as claimed.

Conversely, suppose that $[I]\in\ker\phi$: then, there is a $c\in F$ (where $F$ is the quotient field of $R$) such that $IK=cRK$, i.e., $c^{-1}IK=RK$. Let $S:=D\setminus\{0\}$: then, $RK$ is just the localization $S^{-1}R$, and thus the equality $c^{-1}IK=RK$ can be written as $S^{-1}(c^{-1}I)=S^{-1}R$. Since $I$ is invertible, it is finitely generated; hence there is an $s\in S$ such that $sc^{-1}I\subseteq R$. Thus $sc^{-1}I$ is an integral invertible ideal of $R$ such that $sc^{-1}IK=RK$, i.e., $sc^{-1}I\cap K\neq(0)$. Thus $sc^{-1}I$ is unitary, and $[I]=[sc^{-1}I]\in\Pic_u(R,D)$. The claim is proved.
\end{proof}

\begin{cor}\label{cor:IKprinc}
Let $R$ be a $D$-algebra that is an integral domain, and let $I$ be an ideal of $R$. The following are equivalent:
\begin{enumerate}[(i)]
\item\label{cor:IKprinc:princ} $IK$ is principal;
\item\label{cor:IKprinc:Picu} $[I]\in\Pic_u(R,D)$;
\item\label{cor:IKprinc:cI} there is a $c$ such that $cI$ is unitary and integral.
\end{enumerate}
\end{cor}
\begin{proof}
The equivalence of the the first two conditions follows from Proposition \ref{prop:quoz-picu}, while if \ref{cor:IKprinc:cI} holds then $cIK=cRK$ is principal, so \ref{cor:IKprinc:princ} holds. The implication \ref{cor:IKprinc:Picu} $\Longrightarrow$ \ref{cor:IKprinc:cI} follows from the last part of the proof of Proposition \ref{prop:quoz-picu}.
\end{proof}

\begin{cor}\label{cor:Picu-K}
If $K$ is a field and $R$ is a $K$-algebra, then $\Pic_u(R,K)=(0)$.
\end{cor}
\begin{proof}
By Corollary \ref{cor:IKprinc}, $[I]\in\Pic_u(R,K)$ if and only if $IK=I$ is principal. Thus $\Pic_u(R,K)$ is trivial.
\end{proof}

\begin{cor}
Let $D$ be an integral domain with quotient field $K$, $\XX$ a family of indeterminates. If $R$ is a $D$-algebra such that $D[\XX]\subseteq R\subseteq K[\XX]$, then $\Pic_u(R,D)=\Pic(R)$.
\end{cor}
\begin{proof}
We have $K[\XX]=D[\XX]K\subseteq RK\subseteq K[\XX]$, and thus $RK=K[\XX]$. The ring of polynomials $K[\XX]$ is a unique factorization domain and thus its Picard group is trivial; the claim follows from Proposition \ref{prop:quoz-picu}.
\end{proof}

\section{Pseudo-polynomial algebras}\label{sect:pseudopol}
In order to prove interesting results on the Picard group, we need to further restrict our attention to another class of $D$-algebras.
\begin{defin}
Let $D$ be an integral domain with quotient field $K$, and let $R$ be an extension of $D$. We say that $R$ is \emph{pseudo-polynomial over $D$} if every principal integral ideal of $R$ that is unitary over $D$ is generated by an element of $D$.
\end{defin}

The previous definition can be rewritten in the following way.
\begin{lemma}\label{lemma:unitary-principal}
Let $R$ be a $D$-algebra. Then, $R$ is pseudo-polynomial over $D$ if and only if, for every $r\in R\setminus D$, either $r$ is associated in $R$ to some $d\in D$ or $rR\cap D=(0)$.
\end{lemma}
\begin{proof}
Suppose $R$ is pseudo-polynomial and let $r\in R\setminus D$. If $rR\cap D\neq(0)$, then $I=rR$ is unitary, and thus $I=dR$ is generated by a $d\in D$, i.e., $r$ is associated to $d$. Conversely, if the property in the statement hold and $I=rR$ is unitary over $D$, then $rR\cap D\neq(0)$ and thus $r$ is associated to a $d\in D$, i.e., $I=rD=dR$ is generated by an element of $D$. Thus $R$ is pseudo-polynomial.
\end{proof}

Another interpretation of pseudo-polinomiality is the following: a $D$-algebra $R$ is pseudo-polynomial over $D$ if, for every $a\in D$, all factors of $a$ in $R$ (i.e., all $f\in R$ such that $a\in fR$) are associated to some element of $D$; that is, modulo units, all factors of $a$ in $R$ are actually in $D$.

\begin{lemma}\label{lemma:pseudopol-intersec-princ}
Let $R$ be a pseudo-polynomial $D$-algebra such that $R\cap K=D$. If $I$ is a unitary integral principal ideal of $R$, then $I\cap D$ is principal (over $D$).
\end{lemma}
\begin{proof}
Let $I=rR$ be unitary, and let $J:=I\cap D$. Since $R$ is pseudo-polynomial, $I=dR$ for some $d\in D$. If $x\in J$, then $dx^{-1}\in R\cap K=D$ and thus $x\in dD$, i.e., $d$ generates $J$.
\end{proof}

\begin{oss}
There are $D$-algebras that are not pseudo-polynomial. For example, if $D=\insZ$, $R=\insZ[X,2/X]$, then $I=XR$ is a unitary ideal (since $I\cap\insZ=2\insZ$) but $I$ is not generated by an element of $\insZ$ (since $X/2\notin R$). Note also that $R$ is a retract $\insZ$-algebra when endowed with the evaluation in $1$.
\end{oss}

The study of the pseudo-polinomiality of an extension can always be split into two cases.
\begin{prop}\label{prop:pseudopol-split}
Let $D$ be an integral domain with quotient field $K$, and let $R$ be an integral domain that extends $D$. Then, $R$ is pseudo-polynomial over $D$ if and only if $R$ is pseudo-polynomial over $R\cap K$ and $R\cap K$ is pseudo-polynomial over $D$.
\end{prop}
\begin{proof}
Let $A:=R\cap K$.

Suppose that $R$ is pseudo-polynomial over $D$, and let $f\in R\setminus A$ be such that $fR$ is unitary over $A$. Then also $fR\cap D\neq(0)$ (since $A$ and $D$ have the same quotient field) and thus $f$ is associated to some $d\in D$. Since $D\subseteq A$, it follows that $R$ is pseudo-polynomial over $A$. Moreover, if $f\in A\setminus D$, then $fA$ is unitary over $D$, and thus $f=ud$ for some $d\in D$ and some unit $u$ of $R$. However, $u=f^{-1}d\in R\cap K=A$, and likewise $u^{-1}\in A$: hence, $f$ is associated to $d$ also in $A$. Hence $A$ is pseudo-polynomial over $D$.

Conversely, suppose that $R$ is pseudo-polynomial over $A$ and $A$ is pseudo-polynomial over $D$. Let $f\in R\setminus D$ be such that $fR$ is unitary over $D$: then, either $f\in R\setminus A$ or $f\in A\setminus D$. In the former case,  $fR\cap A\neq(0)$, and thus $f=ua$ for some $a\in A$ and some unit $a\in A$; since $aA\cap D\neq(0)$, we have $a=vd$ for some $d\in D$ and some unit $v$ of $A$. Hence $f=ua=uvd$ and $f$ is associated in $R$ to an element of $D$. If $f\in A\setminus D$, we have $fR\cap D\neq(0)$ and $fA\cap D\neq(0)$, and thus $f=ud$ for some $d\in D$ and some unit $u$ of $A$; hence $u$ is a unit of $R$ and $f$ is associated to $d\in D$ also in $R$. Hence $R$ is pseudo-polynomial over $D$.
\end{proof}

\begin{ex}
When $R$ is an overring of $D$ (i.e., when $R$ is contained in the quotient field of $D$) then every ideal of $R$ is unitary over $D$: therefore, $R$ is pseudo-polynomial if and only if every principal integral ideal of $R$ is generated by an element of $D$. More generally, this criterion holds whenever $R$ is contained in the algebraic closure of the quotient field of $D$ (see Proposition \ref{prop:pseudopol-integral} below).

For example, every localization of $D$ is pseudo-polynomial over $D$. On the other hand, consider $D=\insZ[X]$, and let $R$ be the valuation domain associated to the valuation $v$ defined by
\begin{equation*}
v\left(\sum_{n=0}^ka_nX^n\right)=\inf\{2v_{(2)}(a_n)+3n\},
\end{equation*}
where $v_{(2)}$ is the $2$-adic valuation on $\insZ$. Then, $R$ is a discrete valuation ring, but no element of $D$ generates the maximal ideal of $R$, since no element of $D$ has valuation $1$. Thus $R$ is not pseudo-polynomial over $D$.
\end{ex}

We give two sufficient conditions for an algebra to be pseudo-polynomial.
\begin{prop}\label{prop:pseudopol-pol}
Let $D$ be an integral domain with quotient field $K$, and let $\XX$ be a family of indeterminates over $K$. If $R$ is a $D$-algebra contained in $K[\XX]$ and $R\cap K$ is pseudo-polynomial over $D$, then $R$ is pseudo-polynomial over $D$. In particular, if $R\cap K=D$ then $R$ is pseudo-polynomial over $D$.
\end{prop}
\begin{proof}
By Proposition \ref{prop:pseudopol-split}, we only need to show that $R$ is pseudo-polynomial over $R\cap K$, i.e., we can suppose without loss of generality that $R\cap K=D$. Let $f\in R\setminus D$: then, $f$ is a non-constant polynomial, and thus $fR\subseteq fK[\XX]$ does not contain any constant, i.e., $fR\cap D=(0)$. Thus $R$ is pseudo-polynomial over $D$, as claimed.
\end{proof}

\begin{cor}
Let $D$ be an integral domain. Then, the ring of polynomials $D[X]$ and the ring of integer-valued polynomials $\Int(D)$ are pseudo-polynomial $D$-algebras.
\end{cor}
\begin{proof}
Both rings are contained in $K[X]$, and $D[X]\cap K=\Int(D)\cap K=D$. The claim follows from Proposition \ref{prop:pseudopol-pol}.
\end{proof}

\begin{prop}\label{prop:pseudopol-ufd}
Let $D$ be a unique factorization domain, and let $R$ be an extension of $D$. If every prime element of $D$ is also a prime element of $R$, then $R$ is pseudo-polynomial.
\end{prop}
\begin{proof}
Let $f\in R\setminus D$, and suppose that $a\in fR\cap D$, with $a\neq 0$. Then, $a$ has a prime factorization $a=p_1\cdots p_n$ in $D$; since each $p_i$ is also a prime element of $R$, it follows that $a=p_1\cdots p_n$ is also a prime factorization in $R$. Since $a\in fR$, $f$ is a divisor of $a$, and thus $f$ must be associated to a subproduct $p_{i_1}\cdots p_{i_k}$ of the factorization of $a$. Since $p_{i_1}\cdots p_{i_k}\in D$, the $D$-algebra $R$ is pseudo-polynomial.
\end{proof}

\begin{prop}\label{prop:pseudopol-powseries}
Let $D$ be an integral domain that it is either a unique factorization domain or a Pr\"ufer domain of dimension $1$. Then, $D[[X]]$ is a pseudo-polynomial $D$-algebra.
\end{prop}
\begin{proof}
If $D$ is a unique factorization domain, then for every prime element $p$ of $D$ the quotient $D[[X]]/pD[[X]]$ is isomorphic to $(D/pD)[[X]]$, which is an integral domain; thus, $p$ is also a prime element of $D$ and the claim follows from Proposition \ref{prop:pseudopol-ufd}.

Suppose that $D$ is a Pr\"ufer domain of dimension $1$. Let $f\in D[[X]]$ and suppose $a\in fD[[X]]\cap D$: then, $a=fg$ for some $g\in D[[X]]$. For $h\in Q(D[[X]])$, we denote by $c(h)$ the content of $h$, i.e., the $D$-module generated by the coefficients of $h$. By \cite[Corollary 2.9]{dedekind-mertens-power-series}, we have $(c(f)c(g))^2=c(f)c(g)c(fg)=a\cdot c(f)c(g)$.

Let $f_0$ and $g_0$ be, respectively, the constant term of $f$ and $g$, and let $\tilde{f}:=f/f_0$ and $\tilde{g}:=g/g_0$. Then, $a=f_0g_0$, and $c(\tilde{f})=c(f)/f_0$ and $c(\tilde{g})=c(g)/g_0$; it follows that
\begin{equation*}
(f_0c(\tilde{f})g_0c(\tilde{g}))^2=af_0c(\tilde{f})g_0c(\tilde{g}),
\end{equation*}
i.e, $I^2=I$, where $I:=c(\tilde{f})c(\tilde{g})$. Since $1\in c(\tilde{f})$ and $1\in c(\tilde{g})$, we have $D\subseteq I$; thus, for every maximal ideal $M$, we have $D_M\subseteq ID_M$ and $(ID_M)^2=ID_M$. Since $D_M$ is a one-dimensional valuation domain and $I$ is a fractional ideal of $D$, the only possibility is $I=D_M$ for every $M$, and thus $I=D$; hence also $c(\tilde{f})=c(\tilde{g})=D$, and $c(f)=(f_0)$. Therefore, $f$ is associated to $f_0$ in $D[[X]]$, and $D[[X]]$ is pseudo-polynomial over $D$.
\end{proof}

\begin{ex}
The ring $D[[X]]$ of the power series over $D$ is not always pseudo-polynomial. For example, suppose that $D$ is a two-dimensional valuation ring, with prime ideals $(0)\subsetneq P\subsetneq M$. Let $m\in M\setminus P$ and $p\in P\setminus(0)$. Let
\begin{equation*}
f:=p+pm^{-1}X+pm^{-2}X^2+\cdots=\sum_{i\geq 0}pm^{-i}X^i=\frac{p}{1-m^{-1}X}.
\end{equation*}
Then, $f\in D[[X]]$ since $p\in P\subseteq m^kD$ for every $k$; moreover, $p^2\in fD[[X]]$ since
\begin{equation*}
f\cdot(p-pm^{-1}X)=\frac{p}{1-m^{-1}X}\cdot p(1-m^{-1}X)=p^2.
\end{equation*}
and $p-pm^{-1}X\in D[[X]]$. In particular, $fD[[X]]$ is a integral ideal of $D[[X]]$ that is unitary with respect to $D$.

We claim that $fD[[X]]$ is not generated by any $d\in D$. Indeed, if $h$ is a unit of $D[[X]]$ then its constant term is a unit of $D$; thus, if $fD[[X]]$ is associated in $R$ to some $d\in D$, then $d$ must be associated in $D$ to the constant term of $f$, i.e., to $p$. However,
\begin{equation*}
\frac{f}{p}=\frac{p\sum_{i\geq 0}m^{-i}X^i}{p}=\sum_{i\geq 0}m^{-i}X^i\notin D[[X]]
\end{equation*}
since $m^{-1}\notin D$. Thus $D[[X]]$ is not pseudo-polynomial over $D$.
\end{ex}

\begin{prop}
Let $D\subset R\subset A$ be integral domains, and let $L$ be the quotient field of $R$. If $A$ is pseudo-polynomial over $R$, $R$ is pseudo-polynomial over $D$ and $A\cap L=R$, then $A$ is pseudo-polynomial over $D$.
\end{prop}
\begin{proof}
Let $K$ be the quotient field of $D$. Let $a\in A$ be an element such that $aA\cap K\neq(0)$. Then, $aA\cap L\neq(0)$, and thus there is a $r\in R$ such that $aA=rA$. Therefore, if $t\in aA\cap K$ then $tr^{-1}\in A\cap L=R$, and thus $t\in rR$; in particular, $t\in rR\cap K$, and so $rR\cap K\neq(0)$. Since $R$ is pseudo-polynomial over $D$, there is a $d\in D$ such that $rR=dR$; hence, $aA=rA=rRA=dRA=dA$, and $A$ is pseudo-polynomial over $D$.
\end{proof}

\begin{ex}
The previous proposition does not hold without the hypothesis that $A\cap L=R$. Indeed, let $D=\insZ$, $R=\insZ[X]$ and let $A=V$ be the valuation overring of $R$ induced by the valuation $v$ defined by
\begin{equation*}
v\left(\sum_{n=0}^ka_nX^n\right)=\inf\{2v_{(2)}(a_n)+n\}.
\end{equation*}
Then, $V$ is a discrete valuation ring, and its ideals are generated by the powers $X^n$, which are in $R$; hence, $V$ is pseudo-polynomial over $R$. The above part of the section also implies that $R$ is pseudo-polynomial over $D$.

Let $M$ be the maximal ideal of $V$. Then, $M\cap\insQ\neq(0)$ because $2\in M$; indeed, $M\cap\insQ=2\insZ$. However, $M$ is not generated (over $V$) by any rational number, since $v(2)=2$; thus $V$ is not pseudo-polynomial over $D$.
\end{ex}

Integral extensions are often not pseudo-polynomial.
\begin{prop}\label{prop:pseudopol-integral}
Let $D\subseteq R$ be integral domains, and let $K,L$ be the quotient fields of $D$ and $R$ (respectively). Suppose that the extension $K\subseteq L$ is algebraic. Then, $R$ is pseudo-polynomial over $D$ if and only if every principal ideal of $R$ is generated by an element of $D$.
\end{prop}
\begin{proof}
It is enough to prove that every ideal of $R$ meets $D$. Write $D\subseteq A_1\subseteq A_2\subseteq R$, where $A_1:=R\cap K$ and $A_2$ is the integral closure of $A_1$ in $R$. Then, both $D$ and $A_1$ have quotient field $K$ and both $A_2$ and $R$ have quotient field $L$; thus, every ideal of $R$ meets $A_2$ and every ideal of $A_1$ meets $D$. The claim will thus be proved if every ideal of $A_2$ meets $A_1$.

Let $a\in A_2$: then, $a$ is integral over $A_1$, and thus it has a minimal polynomial $f(X)=f_0+f_1X+\cdots+X^n$. Then, $f_0=-a(f_1+f_2a+\cdots+a^{n-1})$ belongs to both $A_1$ and $aA_2$. and thus $aA_1\cap A_2\neq(0)$. The claim is proved.
\end{proof}

\begin{cor}
Let $R$ be the integral closure of $\insZ$ in a proper extension $L$ of $\insQ$. Then, $R$ is not pseudo-polynomial over $\insZ$.
\end{cor}
\begin{proof}
Since $L\neq\insQ$, there is at least one prime $p$ of $\insZ$ which splits in $R$. Thus, the prime ideals over $p\insZ$ are not generated by elements of $\insZ$. The claim follows from Proposition \ref{prop:pseudopol-integral}.
\end{proof}

\begin{oss}
It is possible for a proper integral extension to be pseudo-polynomial. For example, if $V$ is a valuation domain and $W$ is an extension of $V$ such that the extension of value groups is trivial, then every ideal of $W$ is generated by elements of $V$.
\end{oss}

\section{The local Picard group}\label{sect:picard}
\begin{defin}
Let $R$ be a $D$-algebra and let $\iota:\Pic(D)\longrightarrow\Pic(R)$ be the canonical map. We define the \emph{local Picard group} of $R$ as a $D$-algebra as
\begin{equation*}
\locpic(R,D):=\frac{\Pic(R)}{\iota(\Pic(D))}.
\end{equation*}
Likewise, if $R$ is an extension of $D$ then the \emph{unitary local Picard group} of $D\subseteq R$ is
\begin{equation*}
\locpic_u(R,D):=\frac{\Pic_u(R,D)}{\iota(\Pic(D))}.
\end{equation*}
\end{defin}

\begin{oss}
~\begin{enumerate}
\item Note that if $I$ is an invertible integral ideal of $D$ then $IR$ is integral and unitary. Thus $\iota(\Pic(D))\subseteq\Pic_u(R)$ and $\locpic_u(R,D)$ is well-defined. 
\item From Proposition \ref{prop:quoz-picu} and the basic properties of groups we have
\begin{equation*}
\frac{\locpic(R,D)}{\locpic_u(R,D)}=\frac{\Pic(R)/\iota(\Pic(D))}{\Pic_u(R,D)/\iota(\Pic(D))}\simeq\frac{\Pic(R)}{\Pic_u(R,D)}\simeq\Pic(RK).
\end{equation*}
\item Every ring $R$ can be considered as a $\insZ$-algebra; in this case, the map $\iota$ is just the zero map. Therefore, the Picard group $\Pic(R)$ can also be seen as the local Picard group $\locpic(R,\insZ)$. Likewise, if $F$ is a field and $R$ is an $F$-algebra, $\locpic(R,F)$ is just $\Pic(R)$.
\end{enumerate}
\end{oss}

\begin{ex}\label{ex:DX}
Let $D$ be an integral domain and $R=D[X]$ the polynomial ring over $D$. Then, the canonical map $\Pic(D)\longrightarrow\Pic(D[X])$ is surjective if and only if $D$ is seminormal \cite[Theorem 1.6]{pic-RX-seminormal}, i.e., $\locpic(D[X],D)$ is the trivial group if and only if $D$ is seminormal. More generally, $\locpic(D[X],D)$ is isomorphic to $\Pic(A+XD[X])$, where $A$ is the base ring of $D$ (i.e., $A=\insZ$ if $D$ has characteristic $0$, $A=\ins{F}_p$ is $D$ has characteristic $p>0$) \cite[Theorem 3.8]{picard-DX}, and a similar result holds with more indeterminates.
\end{ex}

\begin{prop}\label{prop:locpic-funct}
Let $D$ be an integral domain. Then, the assignments $R\mapsto\locpic(R,D)$ and $R\mapsto\locpic_u(R,D)$ give rise to functors from the category of integral $D$-algebras (where maps are $D$-algebra homomorphisms) and the category of abelian groups.
\end{prop}
\begin{proof}
Let $\phi:R\longrightarrow R'$ be a map of $D$-algebras. Since $\Pic$ is a functor, $\phi$ induces a map $\phi^\ast:\Pic(R)\longrightarrow\Pic(R')$ sending $\iota_R(\Pic(D))$ to $\iota_{R'}(\Pic(D))$; hence $\phi^\ast$ induces a map $\phi^\sharp:\locpic(R,D)\longrightarrow\locpic(R',D)$. The fact that $\phi\mapsto\phi^\sharp$ respects compositions is seen in the same way.

The proof for the unitary local Picard group is the same.
\end{proof}

\begin{prop}
Let $D\subseteq R\subseteq A$ be integral domains. Then,
\begin{equation*}
\locpic(A,R)\simeq\frac{\locpic(A,D)}{\widetilde{\iota}(\locpic(R,D))}
\end{equation*}
and
\begin{equation*}
\locpic_u(A,R)\simeq\frac{\locpic_u(A,D)}{\widetilde{\iota_u}(\locpic_u(R,D))}
\end{equation*}
where $\widetilde{\iota}:\locpic(R,D)\longrightarrow\locpic(A,D)$ is the map induced by the inclusion $R\subseteq A$ and $\widetilde{\iota_u}$ is the restriction of $\widetilde{\iota}$ to $\locpic_u(R,D)$.
\end{prop}
\begin{proof}
Let $\iota_{DR}:\Pic(D)\longrightarrow\Pic(R)$, $\iota_{RA}:\Pic(R)\longrightarrow\Pic(A)$, $\iota_{DA}:\Pic(D)\longrightarrow\Pic(A)$ be the canonical maps. Then, $\iota_{DA}=\iota_{RA}\circ\iota_{DR}$, and thus in particular $\iota_{RA}(\Pic(R))\supseteq\iota_{DA}(\Pic(D))$. Hence, there is a surjective map
\begin{equation*}
\locpic(A,D)=\frac{\Pic(A)}{\iota_{DA}(\Pic(D))}\longrightarrow\frac{\Pic(A)}{\iota_{RA}(\Pic(R))}=\locpic(A,R),
\end{equation*}
whose kernel is
\begin{equation*}
\frac{\iota_{RA}(\Pic(R))}{\iota_{DA}(\Pic(D))}=\frac{\iota_{RA}(\Pic(R))}{\iota_{RA}\circ\iota_{DR}(\Pic(D))}=\widetilde{\iota}\left(\frac{\Pic(R)}{\iota_{DR}(\Pic(D)}\right)=\widetilde{\iota}(\locpic(R,D)).
\end{equation*}
The claim for $\locpic(A,D)$ is proved. The case of the unitary Picard group is analogous.
\end{proof}

In Example \ref{ex:DX}, the natural map of $\Pic(D)$ into $\Pic(R)$ is not only injective, but give rise to a direct sum decomposition $\Pic(R)\simeq\Pic(D)\oplus\locpic(R,D)$ \cite[Section 2]{picard-DX}; this is a more general feature of retract $D$-algebras, and can be proved essentially in the same way.
\begin{prop}\label{prop:locpic-directsum}
Let $D$ be an integral domain and let $R$ be a retract $D$-algebra. Then:
\begin{enumerate}[(a)]
\item the canonical map $\iota:\Pic(D)\longrightarrow\Pic(R)$ is injective;
\item $\iota(\Pic(D))$ is a direct summand of $\Pic(R)$ and of $\Pic_u(R,D)$;
\item $\Pic(R)\simeq\Pic(D)\oplus\locpic(R,D)$;
\item $\Pic_u(R,D)\simeq\Pic(D)\oplus\locpic_u(R,D)$.
\end{enumerate}
\end{prop}
\begin{proof}
Let $i$ be the inclusion of $D$ into $R$. The composition $\epsilon\circ i$ is the identity on $D$; since $A\mapsto\Pic(A)$ is a functor, it follows that $\epsilon^\ast\circ i^\ast=\epsilon^\ast\circ\iota$ is the identity on $\Pic(D)$. Therefore, $\iota$ is injective and the exact sequence
\begin{equation*}
0\longrightarrow\ker(\epsilon^\ast)\longrightarrow\Pic(R)\xrightarrow{~~\epsilon^\ast~~}\Pic(D) \longrightarrow 0
\end{equation*}
splits. The kernel of $\epsilon^\ast$ is isomorphic to the quotient between $\Pic(R)$ and $\iota(\Pic(D))$, and thus by definition is isomorphic to $\locpic(R,D)$. Hence, $\Pic(R)\simeq\Pic(D)\oplus\locpic(R,D)$.

The reasoning for $\locpic_u(R,D)$ is the same.
\end{proof}

\section{Localization of the local Picard group}\label{sect:jaffard}
We aim to study the local Picard group through the lens of localization and of extension by Jaffard overrings, as was done for the Picard group of the ring of integer-valued polynomials in \cite{PicInt}. We shall follow the same method of the proofs therein, which are generalizations of the methods given in \cite[Chapter VIII]{intD}.

The following theorem corresponds to \cite[Proposition VIII.1.6]{intD} and \cite[Proposition 4.3]{PicInt}.
\begin{teor}\label{teor:exseq-inttype}
Let $D$ be an integral domain and let $R$ be a pseudo-polynomial retract $D$-algebra. Let $\Theta$ be a complete family of flat overrings of $D$. Then, there are exact sequences
\begin{equation}\label{eq:exseq-inttype}
0\longrightarrow\Pic(D,\Theta)\longrightarrow\Pic(R)\xrightarrow{~~\pi_\Theta~~}\prod_{T\in\Theta}\Pic(RT).
\end{equation}
and
\begin{equation}\label{eq:exseq-inttype-u}
0\longrightarrow\Pic(D,\Theta)\longrightarrow\Pic_u(R,D)\xrightarrow{~~\pi_\Theta~~}\prod_{T\in\Theta}\Pic_u(RT,R).
\end{equation}
\end{teor}
\begin{proof}
We first show the result for $\Pic(R)$.

The map $\Pic(D,\Theta)\longrightarrow\Pic(R)$ is the restriction of the extension map $\iota$, which is injective by Proposition \ref{prop:locpic-directsum}, and thus it is itself injective. By construction, if $[I]\in\Pic(D,\Theta)$ then $IT$ is principal for every $T\in\Theta$, and thus $IRT=ITR$ is principal; thus, the kernel of $\pi_\Theta$ contains $\iota(\Pic(D,\Theta))$.

Suppose now that $[I]\in\ker\pi_\Theta$; then, $IT$ is principal for every $T\in\Theta$ and thus $IK$ is principal. By the proof of Proposition \ref{prop:quoz-picu}, $[I]\in\Pic_u(R,D)$, and thus without loss of generality we can suppose that $I$ is unitary and integral. Then, $(I\cap D)T=T$ for all but finitely many elements of $\Theta$, say $T_1,\ldots,T_n$. By Lemma \ref{lemma:pseudopol-intersec-princ}, for each $i$ there is an $x_i\in IT_i$ such that $IT_i=x_iRT_i$. The ideal $L_i:=x_iT_i\cap D$ is finitely generated over $D$ (since $T_i$ is a Jaffard overring \cite[Lemma 5.9]{starloc}) and $L_iT_i=x_iT_i$. Therefore, the ideal $L:=L_1+\cdots+L_n$ is a finitely generated ideal of $D$; moreover, $LT=T$ if $T\in\Theta\setminus\{T_1,\ldots,T_n\}$ and $LT_i=L_iT_i=x_iT_i=IT_i$, and thus $L$ is locally principal. Therefore, $L$ is an invertible ideal such that $LT$ is principal for every $T\in\Theta$ (thus, $L\in\Pic(D,\Theta)$) and $LRT=IRT$ for every $T\in\Theta$. As the family $\Theta$ is complete, we have $R=\bigcap\{RT\mid T\in\Theta\}$; therefore, the map $\star:Z\mapsto\bigcap\{ZRT\mid T\in\Theta\}$ is a star operation on $R$ (see for example \cite[\textsection 32]{gilmer}), and $I$ and $LR$ are invertible ideals of $R$. Thus
\begin{equation*}
I=I^\star=\bigcap_{T\in\Theta}IT=\bigcap_{T\in\Theta}LRT=(LR)^\star=LR,
\end{equation*}
i.e., $[I]=\iota([L])$. Thus $\ker\pi_\Theta\subseteq\iota(\Pic(D,\Theta))$, as claimed.

The result for $\Pic_u(R,D)$ follows by restricting the previous reasoning to unitary ideals and noting that the extension of a unitary integral ideal is still unitary and integral.
\end{proof}

Putting more hypothesis on $\Theta$, we are able to get stronger statements. Lemma \ref{lemma:fg-intersec} below is a variant of \cite[Lemma 5.9]{starloc}.
\begin{lemma}\label{lemma:flatintersect}
Let $D$ be an integral domain, $R$ a $D$-algebra with quotient field $L$ and $T$ a flat overring of $D$. If $X_1,X_2$ are $R$-submodules of $L$, then $(X_1\cap X_2)T=X_1T\cap X_2T$.
\end{lemma}
\begin{proof}
If $T$ is a flat overring of $D$, then $TR$ is a flat overring of $R$, and since $X_1,X_2$ are $R$-modules,
\begin{equation*}
(X_1\cap X_2)T=(X_1\cap X_2)RT=X_1RT\cap X_2RT=X_1T\cap X_2T,
\end{equation*}
as claimed.
\end{proof}

\begin{lemma}\label{lemma:fg-intersec}
Let $D$ be an integral domain and let $R$ be an extension of $D$; let $T$ be a Jaffard overring of $D$. Let $J$ be a unitary ideal of $RT$. If $J$ is finitely generated over $RT$, then $J\cap R$ is finitely generated over $R$.
\end{lemma}
\begin{proof}
Let $T^\perp$ be the orthogonal to $T$ with respect to $D$, and let $I:=J\cap R$. Using Lemma \ref{lemma:flatintersect}, we have
\begin{equation*}
IT^\perp=(J\cap R)T^\perp=JT^\perp\cap T^\perp=JRTT^\perp\cap T^\perp=JRK\cap T^\perp.
\end{equation*}
Since $J$ is unitary, $J\cap K\neq(0)$; thus, $JRK=(JK)R=KR$ and $IT^\perp=T^\perp$. Hence there is a finitely generated ideal $I_0\subseteq I$ such that $I_0T^\perp=T^\perp$.

Let $x_1,\ldots,x_m$ be the generators of $J$. Since $IT=(J\cap R)T=JT\cap T=J$, for each $i$ there is a finitely generated ideal $I_i\subseteq I$ such that $x_i\in I_iT$; then, $L:=I_0+I_1+\cdots+I_m$ is a finitely generated ideal contained in $I$ such that $LT=J=IT$ and $LT^\perp=T^\perp=IT^\perp$. It follows that $L=I$, i.e., $I=J\cap R$ is finitely generated.
\end{proof}

The following is an analogue of Theorems 4.4 and 4.7 of \cite{PicInt}.
\begin{teor}\label{teor:exseq-inttype-jaffard}
Let $D$ be an integral domain and let $R$ be a pseudo-polynomial retract $D$-algebra. Let $\Theta$ be a Jaffard family of $D$. Then, there is an exact sequence
\begin{equation*}
0\longrightarrow\Pic(D,\Theta)\longrightarrow\Pic_u(R,D)\xrightarrow{~~\pi_\Theta~~}\bigoplus_{T\in\Theta}\Pic_u(RT,T)\longrightarrow 0,
\end{equation*}
and
\begin{equation*}
\locpic_u(R,D)\simeq\bigoplus_{T\in\Theta}\locpic_u(RT,T).
\end{equation*}
\end{teor}
\begin{proof}
By Theorem \ref{teor:exseq-inttype}, to prove the first claim it is enough to show that the range of $\pi_\Theta$ is the direct sum. Indeed, if $[I]\in\Pic_u(R,D)$, then by Corollary \ref{cor:IKprinc} we can suppose without loss of generality that $I$ is unitary and integral. Hence, $(I\cap D)T=T$ for all but finitely many $T\in\Theta$, and thus the range of $\pi_\Theta$ is contained in the direct sum.

To prove the converse, we need to show that, given a fixed $T\in\Theta$ and a $[J]\in\Pic_u(RT,T)$, there is an $[I]\in\Pic_u(R,D)$ such that $[IT]=[J]$ and $[IS]=[RS]$ for all $S\in\Theta\setminus\{T\}$. By Corollary \ref{cor:IKprinc}, we can suppose that $J$ is integral and unitary. By Lemma \ref{lemma:fg-intersec}, $I:=J\cap R$ is finitely generated over $R$. We claim that $I$ satisfies the previous conditions.

Indeed, $IT=(J\cap R)T=JT\cap RT=JT$, while $IS=(J\cap R)S=JS\cap RS=JTS\cap RS=JK\cap RS=RS$ since $J\cap K\neq(0)$ and thus $1\in JK$. To show that $I$ is invertible, let $M$ be a maximal ideal of $R$. If $M\cap D=(0)$, then $R_M$ contains $RK$ and thus $RT$, for every $T\in\Theta$; hence, $IR_M$ is principal since so is $IRT$. If $M\cap D=P\neq(0)$, then $R_M$ contains $D_P$, and thus $R_M\supseteq IS$, where $S$ is the member of $\Theta$ such that $PS\neq S$. Thus $IR_M$ is principal since $IRS$ is principal. Therefore, $I$ is locally principal and thus invertible.

Therefore, the direct sum is in the image of $\pi_\Theta$, and the sequence in the statement is exact.

Consider now the commutative diagram 
\begin{equation*}
\begin{tikzcd}
0\arrow[r] & \Pic(D,\Theta) \arrow[r]\arrow[d,equal] & \Pic(D) \arrow[r]\arrow[d,"\iota_D"] & \displaystyle{\bigoplus_{T\in\Theta}\Pic(T)}\arrow[d,"\iota_\Theta"] \arrow[r] & 0\\
0\arrow[r] & \Pic(D,\Theta) \arrow[r] & \Pic_u(R,D)\arrow[r] & \displaystyle{\bigoplus_{T\in\Theta}\Pic_u(RT,T)}\arrow[r] & 0
\end{tikzcd}
\end{equation*}
where $\iota_D$ and $\iota_\Theta$ are the natural maps. Since the leftmost vertical map is the equality, the snake lemma gives an isomorphism between the cokernel of $\iota_D$ (namely, $\locpic_u(R,D)$) and the cokernel of $\iota_\Theta$ (namely, the direct sum of the $\locpic_u(RT,T)$). The claim follows.
\end{proof}

\begin{cor}\label{cor:exseq-inttype-jaffard}
Let $D$ be an integral domain with quotient field $K$, and let $R$ be a pseudo-polynomial retract $D$-algebra such that $\Pic(RK)=(0)$. Let $\Theta$ be a Jaffard family of $D$. Then, there is an exact sequence
\begin{equation*}
0\longrightarrow\Pic(D,\Theta)\longrightarrow\Pic(R)\xrightarrow{~~\pi_\Theta~~}\bigoplus_{T\in\Theta}\Pic(RT)\longrightarrow 0,
\end{equation*}
and
\begin{equation*}
\locpic(R,D)\simeq\bigoplus_{T\in\Theta}\locpic(RT,T).
\end{equation*}
\end{cor}
\begin{proof}
If $\Pic(RK)=(0)$, then $\Pic_u(R,D)=\Pic(R)$ and $\locpic_u(R,D)=\locpic(R,D)$ (and the same for $T$, since $\Pic(RTK)=\Pic(RK)$). The claim follows from Theorem \ref{teor:exseq-inttype-jaffard}.
\end{proof}

\begin{oss}
Without the hypothesis $\Pic(RK)=(0)$, Corollary \ref{cor:exseq-inttype-jaffard} does not hold. Indeed, in that case we would have a commutative diagram
\begin{equation*}
\begin{tikzcd}
0\arrow[r] & \Pic(D,\Theta) \arrow[r]\arrow[d,equal] & \Pic_u(R,D)\arrow[r]\arrow[d] & \displaystyle{\bigoplus_{T\in\Theta}\Pic_u(RT,T)}\arrow[r]\arrow[d] & 0\\
0\arrow[r] & \Pic(D,\Theta) \arrow[r] & \Pic(R)\arrow[r] & \displaystyle{\bigoplus_{T\in\Theta}\Pic(RT)}\arrow[r] & 0
\end{tikzcd}
\end{equation*}
and an application of the snake lemma would give an isomorphism
\begin{equation*}
\Pic(RK)\simeq\bigoplus_{T\in\Theta}\Pic(RK),
\end{equation*}
which does not hold, in general.
\end{oss}

The following corollaries are special cases of Theorem \ref{teor:exseq-inttype-jaffard} and Corollary \ref{cor:exseq-inttype-jaffard}.
\begin{cor}
Let $D$ be an integral domain and let $\Theta$ be a Jaffard family of $D$. Let $\mathbf{X}$ be a family of independent indeterminates over $D$. Then,
\begin{equation*}
\locpic(D[\mathbf{X}],D)\simeq\bigoplus_{T\in\Theta}\locpic(T[\mathbf{X}],T).
\end{equation*}
\end{cor}
\begin{proof}
The polynomial ring $D[\mathbf{X}]$ is a pseudo-polynomial retract $D$-algebra. Since $D[\XX]K=K[\XX]$ is a unique factorization domain, we can apply Corollary \ref{cor:exseq-inttype-jaffard}.
\end{proof}

\begin{cor}
Let $D$ be an integral domain and let $\Theta$ be a Jaffard family of $D$. Let $E\subseteq K$ be a subset such that $dE\subseteq D$ for some $d\neq 0$. Then,
\begin{equation*}
\locpic(\Int(E,D),D)\simeq\bigoplus_{T\in\Theta}\locpic(\Int(E,T),T)
\end{equation*}
and
\begin{equation*}
\locpic(\Int(D),D)\simeq\bigoplus_{T\in\Theta}\locpic(\Int(T),T)
\end{equation*}
\end{cor}
\begin{proof}
Note that $\Int(E,D)\simeq\Int(dE,D)$, and thus we can suppose without loss of generality that $E\subseteq D$.

The ring $\Int(E,D)$ is a retract $D$-algebra since the evaluation in any $d\in E$ is a retract. Moreover, $D[X]\subseteq\Int(E,D)$, and thus $\Int(E,D)K=K[X]$; therefore, $\Pic(\Int(E,D)K)=(0)$, so that $\Pic_u(\Int(E,D),D)=\Pic(\Int(E,D))$ and $\locpic_u(\Int(E,D),D)=\locpic(\Int(E,D),D)$. Finally, $\Int(E,D)\cap K=D$ and thus $D$ is pseudo-polynomial by Proposition \ref{prop:pseudopol-pol}. By Theorem \ref{teor:exseq-inttype-jaffard}, we have
\begin{equation*}
\locpic(\Int(E,D),D)\simeq\bigoplus_{T\in\Theta}\locpic(\Int(E,D)T,T).
\end{equation*}
The equality $\Int(E,D)T=\Int(E,T)$ follows as in \cite[Section 3]{PicInt}.

If $E=D$, then $\Int(E,D)=\Int(D)$ and $\Int(D)T=\Int(T)$. The claim is proved.
\end{proof}

Let $x\in D$. The \emph{Bhargava ring} of $D$ with respect to $x$ is \cite{yeramian-bhargava}
\begin{equation*}
\mathbb{B}_x(D):=\{f\in K[X]\mid f(aX+x)\in D[X]\text{~for all~}a\in D\}.
\end{equation*}
\begin{cor}
Let $D$ be an integral domain and let $\Theta$ be a Jaffard family of $D$.Then,
\begin{equation*}
\locpic(\mathbb{B}_x(D),D)\simeq\bigoplus_{T\in\Theta}\locpic(\mathbb{B}_x(T),T).
\end{equation*}
\end{cor}
\begin{proof}
The Bhargava ring $\mathbb{B}_x(D)$ satisfies $D[X]\subseteq\mathbb{B}_x(D)\subseteq K[X]$, and $\mathbb{B}_x(D)\cap K=D$; therefore, it is pseudo-polynomial and $\Pic(\mathbb{B}_x(D)K)=(0)$. Moreover, $\mathbb{B}_x(D)\subseteq\Int(x,D)$ and thus $\mathbb{B}_x(D)$ is also a retract $D$-algebra. By Corollary \ref{cor:exseq-inttype-jaffard}, we have $\locpic(\mathbb{B}_x(D),D)\simeq\bigoplus_{T\in\Theta}\locpic(\mathbb{B}_x(D)T,T)$, and we need to show that $\mathbb{B}_x(D)T=\mathbb{B}_x(T)$.

Let $T^\perp$ be the orthogonal of $T$ with respect to $D$. By \cite[Lemma 1.1]{bhargava-spec}, since $T$ and $T^\perp$ are sublocalizations and $D=T\cap T^\perp$ we have $\mathbb{B}_x(D)=\mathbb{B}_x(T)\cap\mathbb{B}_x(T^\perp)$; by Lemma \ref{lemma:flatintersect} it follows that
\begin{equation*}
\mathbb{B}_x(D)T=(\mathbb{B}_x(T)\cap\mathbb{B}_x(T^\perp))T=\mathbb{B}_x(T)T\cap\mathbb{B}_x(T^\perp)T.
\end{equation*}
Since $\mathbb{B}_x(T)$ is a $T$-algebra we have $\mathbb{B}_x(T)T=\mathbb{B}_x(T)\subseteq K[X]$. Moreover, $\mathbb{B}_x(T^\perp)T=\mathbb{B}_x(T^\perp)T^\perp T=\mathbb{B}_x(T^\perp)K=K[X]$. Thus $\mathbb{B}_x(T)T\subseteq\mathbb{B}_x(T^\perp)T$ and $\mathbb{B}_x(D)T=\mathbb{B}_x(T)$. The claim is proved.
\end{proof}

\begin{cor}\label{cor:locfin-dim1}
Let $D$ be a locally finite one-dimensional domain, and let $R$ be a pseudo-polynomial retract $D$-algebra. Then,
\begin{equation*}
\locpic_u(R,D)\simeq\bigoplus_{M\in\Max(D)}\Pic_u(RD_M,D_M),
\end{equation*}
and
\begin{equation*}
\Pic_u(R,D)\simeq\Pic(D)\oplus\bigoplus_{M\in\Max(D)}\Pic_u(RD_M,D_M),
\end{equation*}
\end{cor}
\begin{proof}
The first isomorphism follows from Theorem \ref{teor:exseq-inttype-jaffard} using the family $\Theta:=\{D_M\mid M\in\Max(D)\}$ (which is a Jaffard family since $D$ is one-dimensional and locally finite), and the fact that $\Pic(D_M)=(0)$ since each $D_M$ is local. The second isomorphism now follows from Proposition \ref{prop:locpic-directsum}.
\end{proof}

\begin{cor}\label{cor:PicD0}
Let $D$ be an integral domain and let $\Theta$ be a Jaffard family of $D$; let $R$ be a pseudo-polynomial retract $D$-algebra. If $\Pic(D)=(0)$, then
\begin{equation*}
\Pic_u(R,D)\simeq\bigoplus_{T\in\Theta}\Pic_u(RT,T).
\end{equation*}
\end{cor}
\begin{proof}
If $\Pic(D)=(0)$, then $\Pic(T)=(0)$ for every Jaffard overring $T$; thus, $\locpic_u(R,D)=\Pic_u(R,D)$ and $\locpic_u(RT,T)=\Pic_u(RT,T)$.
\end{proof}

\begin{cor}
Let $D$ be a locally finite Pr\"ufer domain of dimension $1$. Then,
\begin{equation*}
\Pic_u(D[[X]],D)\simeq\Pic(D)\oplus\bigoplus_{M\in\Max(D)}\Pic_u(D[[X]]D_M,D_M),
\end{equation*}
\end{cor}
Note that $D[[X]]D_M$ is \emph{not} the ring $D_M[[X]]$ of power series over $D_M$: for example, $\sum_n X^n/3^n$ belongs to $\insZ_{(2)}[[X]]$ but not to $\insZ[[X]]\insZ_{(2)}$.
\begin{proof}
The ring $D[[X]]$ is pseudo-polynomial over $D$ by Proposition \ref{prop:pseudopol-powseries} and a retract $D$-algebra (with $\epsilon$ being the evaluation in $0$). The claim now follows from Corollary \ref{cor:locfin-dim1}.
\end{proof}

\begin{oss}
Theorems \ref{teor:exseq-inttype} and \ref{teor:exseq-inttype-jaffard} do not hold for general $D$-algebras. For example, let $D=\insZ$ and let $R$ be the integral closure of $\insZ$ in a finite extension $L$ of $\insQ$. Note that $\Pic_u(R,\insZ)=\Pic(R)$ since $RK=L$. Let $\Theta$ be the family of localizations of $\insZ$ at the maximal ideals. Then, for every $T=\insZ_{(p)}\in\Theta$, the ring $RT$ is semilocal (its maximal ideals correspond to the maximal ideals of $R$ over $(p)$, which are finite since $[L:\insQ]<\infty$), and thus $\Pic(RT)=(0)$ for every $T$, and thus also $\locpic(RT,T)=(0)$. On the other hand, $\Pic(D,\Theta)=(0)$, and thus \eqref{eq:exseq-inttype} becomes
\begin{equation*}
0\longrightarrow 0\longrightarrow\Pic(R)\longrightarrow0.
\end{equation*}
If $R$ does not have unique factorization, $\Pic(R)\neq(0)$ and thus the sequence is not exact. Likewise, $\locpic(R,D)=\Pic(R)\neq(0)$, and thus the isomorphism $\locpic(R,D)\simeq\bigoplus_{T\in\Theta}\locpic(RT,T)$ is not true.
\end{oss}

\section{Pre-Jaffard families}\label{sect:preJaff}
The results in the previous section only deal with Jaffard families. As done in \cite{PicInt}, under some hypothesis we can extend the results to the more general case of pre-Jaffard families.

We start with an analogue of \cite[Proposition 7.2]{PicInt}, of which we follow the proof.
\begin{lemma}\label{lemma:picTalpha-surj}
Let $D$ be an integral domain and $\Theta$ be a pre-Jaffard family of $D$; let $\{T_\alpha\}$ be the derived sequence associated to $\Theta$. Let $R$ be a pseudo-polynomial retract $D$-algebra. Then, the extension map $\Pic_u(R,D)\longrightarrow\Pic_u(RT_\alpha,T_\alpha)$ is surjective.
\end{lemma}
\begin{proof}
By induction on $\alpha$. If $\alpha=1$, let $\mathcal{L}$ be the lattice of Jaffard overrings of $D$. By the proof of \cite[Proposition 6.1]{PicInt}, we have $T=\bigcup\{S\mid S\in\mathcal{L}\}$; using the same proof of \cite[Lemma 5.2]{PicInt}, it follows that $RT=\bigcup\{RS\mid S\in\mathcal{L}\}$. Since for $S\in\Theta\setminus\{T\}$ the map $\Pic_u(R,D)\longrightarrow\Pic_u(RS,S)$ is surjective (Theorem \ref{teor:exseq-inttype-jaffard}), the same reasoning of the proof of \cite[Lemma 5.1]{PicInt} shows that also $\Pic_u(R,D)\longrightarrow\Pic_u(RT,T)$ is surjective.

If $\alpha$ is a limit ordinal, the claim follows in the same way since $\bigcup_{\gamma<\alpha} T_\gamma=T_\alpha$ \cite[Lemma 7.1]{PicInt} and thus $\bigcup_{\gamma<\alpha} RT_\gamma=RT_\alpha$; hence we can apply \cite[Lemma 5.1]{PicInt}. If $\alpha=\gamma+1$ is a successor ordinal, then the map $\Pic_u(R,D)\longrightarrow\Pic_u(RT_\alpha,T_\alpha)$ factors as
\begin{equation*}
\Pic_u(R,D)\longrightarrow\Pic_u(RT_\gamma,T_\gamma)\longrightarrow\Pic_u(RT_\alpha,T_\alpha);
\end{equation*}
the first map is surjective by inductive hypothesis, while the second one is surjective since we can apply the case $\alpha=1$ to the $T_\gamma$-algebra $RT_\gamma$.
\end{proof}

The following result is the analogue of Proposition 6.2 and Theorem 6.4 of \cite{PicInt}. We premise a lemma, that was implicitly used in the proof of \cite[Proposition 6.2]{PicInt}.
\begin{lemma}\label{lemma:numfin-nonprinc}
Let $\Theta$ be a weak Jaffard family of $D$ pointed at $T_\infty$. Let $I$ be a finitely generated ideal of $D$ such that $IT_\infty$ is principal. Then, there are only finitely many $T\in\Theta$ such that $IT$ is not principal.
\end{lemma}
\begin{proof}
Without loss of generality we can suppose that $I\subseteq D$. Let $\Lambda$ be the set of all $T\in\Theta$ such that $IT$ is not principal.

Let $I=(x_1,\ldots,x_n)$. Suppose that $IT_\infty=fT_\infty$: then, there are $t_1,\ldots,t_n\in\ T_\infty$ such that $f=x_1t_1+\cdots+x_nt_n$. Consider the set $\Omega:=\B(t_1,\ldots,t_n,fx_1^{-1},\ldots,fx_n^{-1})\subseteq\Theta$ of all elements of $\Theta$ that contain each $t_i$ and each $fx_i^{-1}$: then, $\Omega$ is an open set with respect to the Zariski topology, and $T\in\Omega$ if and only if $IT=fT$. In particular, $T_\infty\in\Omega$; thus, $\Lambda_0:=\Theta\setminus\Omega$ is a closed set of $\Theta$ not containing $T_\infty$, and $\Lambda\subseteq\Lambda_0$.

Since $\Theta$ is compact in the Zariski topology and $\Lambda_0$ is closed, $\Lambda_0$ is compact. Let $A:=\bigcap\{T\mid T\in\Lambda_0\}$: then, each $T$ is a flat overring of $A$, and it is also a Jaffard overring of $A$ since each such $T$ is a Jaffard overring of $D$. If $P$ is a prime ideal of $D$ such that $PT=T$ for every $T\in\Lambda_0$ then
\begin{equation*}
AD_P=\left(\bigcap_{T\in\Lambda_0}T\right)D_P=\bigcap_{T\in\Lambda_0}TD_P=K
\end{equation*}
using \cite[Corollary 5]{compact-intersections}. Thus, $\Lambda_0$ is a Jaffard family of $A$, and in particular it is locally finite. Consider $IA$ and the extensions $IAT$ for $T\in\Lambda$: then, $IAT=IT$ is not principal for such $T$, and thus $IAT\neq T$. By local finiteness, $\Lambda_0$ must be finite, as claimed.
\end{proof}

\begin{prop}\label{prop:weakJaff}
Let $D$ be an integral domain and let $R$ be a pseudo-polynomial retract $D$-algebra. Let $\Theta$ be a weak Jaffard family of $D$ pointed at $T_\infty$. Let $\pi_\Theta:\Pic_u(R,D)\longrightarrow\prod\{\Pic_u(RT,T)\mid T\in\Theta\}$ be the extension map and let $\Delta$ be its cokernel. Then, there are exact sequences
\begin{equation*}
0\longrightarrow\bigoplus_{T\in\Theta\setminus\{T_\infty\}}\Pic_u(RT,T)\longrightarrow\Delta\longrightarrow\Pic_u(RT_\infty,T_\infty)\longrightarrow 0.
\end{equation*}
and
\begin{equation*}
0\longrightarrow\bigoplus_{T\in\Theta\setminus\{T_\infty\}}\locpic_u(RT,T)\longrightarrow\locpic_u(R,D)\longrightarrow\locpic_u(RT_\infty,T_\infty)\longrightarrow 0.
\end{equation*}
\end{prop}
\begin{proof}
The inclusion of $\Delta$ into the direct sum $\prod_{T\in\Theta}\Pic_u(RT,T)$ induces a projection map $\pi':\Delta\longrightarrow\Pic_u(RT_\infty,T_\infty)$, which is surjective since it factorizes the surjective extension map $\Pic_u(R,D)\longrightarrow\Pic_u(RT_\infty,T_\infty)$.

The kernel of $\pi'$ contains exactly the extensions of the classes $[I]\in\Pic_u(R,D)$ such that $I$ becomes principal in each $RT$, for $T\in\Theta\setminus\{T_\infty\}$. Using Theorem \ref{teor:exseq-inttype-jaffard}, we obtain that the direct sum $\bigoplus\{\Pic_u(RT,T)\mid T\in\Lambda\}$ is contained in the kernel; conversely, if $[I]\in\ker\pi'$ then $IT_\infty=fT_\infty$ for some $f$ and thus $I$ is not principal for only finitely many $T\in\Theta$ (Lemma \ref{lemma:numfin-nonprinc}); thus the kernel is contained in the direct sum. Therefore $\ker\pi'=\bigoplus\{\Pic_u(RT,T)\mid T\in\Lambda\}$, and the first sequence in the statement is exact.

The diagram
\begin{equation*}
\begin{tikzcd}
0\arrow[r] & \displaystyle{\bigoplus_{T\in\Theta\setminus\{T_\infty\}}}\Pic(T) \arrow[r]\arrow[d] & \displaystyle{\frac{\Pic(D)}{\Pic(D,\Theta)}} \arrow[r]\arrow[d] & \Pic(T_\infty)\arrow[d] \arrow[r] & 0\\
0\arrow[r] & \displaystyle{\bigoplus_{T\in\Theta\setminus\{T_\infty\}}}\Pic_u(RT,T) \arrow[r] & \Delta\arrow[r] & \Pic_u(RT_\infty,T_\infty)\arrow[r] & 0
\end{tikzcd}
\end{equation*}
commutes, the vertical maps are injective, and the horizontal rows are exact (the top one by \cite[Lemma 6.3]{PicInt}, the bottom one by the previous part of the proof). Applying the snake lemma, we obtain an exact sequence
\begin{equation*}
0\longrightarrow\bigoplus_{T\in\Theta\setminus\{T_\infty\}}\locpic_u(RT,T)\longrightarrow\frac{\Delta}{\Pic(D)/\Pic(D,\Theta)}\longrightarrow\locpic_u(RT_\infty,T_\infty)\longrightarrow 0,
\end{equation*}
that becomes the one in the statement by noting that
\begin{equation*}
\frac{\Delta}{\Pic(D)/\Pic(D,\Theta)}=\frac{\Pic_u(R,D)/\Pic(D,\Theta)}{\Pic(D)/\Pic(D,\Theta)}\simeq\frac{\Pic_u(R,D)}{\Pic(D)}=\locpic_u(R,D)
\end{equation*}
by definition.
\end{proof}

The following is analogous to Theorem 7.3 of \cite{PicInt}.
\begin{teor}\label{teor:preJaff}
Let $D$ be an integral domain and let $R$ be a pseudo-polynomial retract $D$-algebra. Let $\Theta$ be a pre-Jaffard family of $D$, and let $\{T_\alpha\}$ be the derived series of $D$. Fix an ordinal $\alpha$ and suppose that $\locpic_u(RT,T)$ is a free group for each $T\in\Theta\setminus\njaff^\alpha(\Theta)$. Then, there is an exact sequence
\begin{equation*}
0\longrightarrow\bigoplus_{T\in\Theta\setminus\njaff^\alpha(\Theta)}\locpic_u(RT,T)\longrightarrow\locpic_u(R,D)\longrightarrow \locpic_u(RT_\alpha,T_\alpha)\longrightarrow 0.
\end{equation*}
\end{teor}
\begin{proof}
By induction on $\alpha$: if $\alpha=1$ then $\Theta$ is a weak Jaffard family and the claim follows from Proposition \ref{prop:weakJaff}.

If $\alpha=\gamma+1$ is a successor ordinal, then we have a commutative diagram
\begin{equation}\label{eq:diagpreJaff}
\begin{tikzcd}
0\arrow[r] & \displaystyle{\bigoplus_{T\in\Theta\setminus\njaff^\gamma(\Theta)}\locpic_u(RT,T)}\arrow[r]\arrow[d,"f"] &\locpic_u(R,D)\arrow[r]\arrow[d,equal] &\locpic_u(RT_\gamma,T_\gamma)\arrow[r]\arrow[d,"g"] & 0\\
0\arrow[r] & L\arrow[r] &\locpic_u(R,D)\arrow[r] &\locpic_u(RT_\alpha,T_\alpha)\arrow[r] & 0,
\end{tikzcd}
\end{equation}
where $L$ is the kernel of $\locpic_u(R,D)\longrightarrow\locpic_u(RT_\alpha,T_\alpha)$. The two rows are exact: the first one by induction, the second one by definition and by Lemma \ref{lemma:picTalpha-surj}. The snake lemma and the fact that the middle vertical arrow is the identity give an exact sequence 
\begin{equation*}
0\longrightarrow \bigoplus_{T\in\Theta\setminus\njaff^\gamma(\Theta)}\locpic_u(RT,T)\longrightarrow L\longrightarrow \bigoplus_{T\in\njaff^\gamma(\Theta)\setminus\njaff^\alpha(\Theta)}\locpic_u(RT,T)\longrightarrow 0,
\end{equation*}
which splits since $\locpic_u(RT,T)$ is free for $T\in\njaff^\gamma(\Theta)\setminus\njaff^\alpha(\Theta)$, by hypothesis. The claim follows reading the second row of \eqref{eq:diagpreJaff}.

If $\alpha$ is a limit ordinal, and $L_\gamma$ is the kernel of the surjective map $\locpic_u(R,D)\longrightarrow\locpic_u(RT_\gamma,T_\gamma)$, then $\{L_\gamma\}_{\gamma<\alpha}$ is a chain of free subgroups such that every element is a direct summand of the following ones and whose union is $L_\alpha$. Therefore, by \cite[Lemma 5.6]{SP-scattered} (or \cite[Chapter 3, Lemma 7.3]{fuchs-abeliangroups}) $L_\alpha$ is the direct sum of the quotients 
\begin{equation*}
\frac{L_{\gamma+1}}{L_\gamma}\simeq\frac{\locpic_u(RT_{\gamma+1},T_{\gamma+1})}{\locpic_u(RT_\gamma+,T_\gamma)}\simeq\bigoplus_{T\in\njaff^{\gamma+1}(\Theta)\setminus\njaff^\gamma(\Theta)}\locpic_u(RT,T).
\end{equation*}
Hence, $L_\alpha\simeq\bigoplus\{\locpic_u(RT,T)\mid T\in\Theta\setminus\njaff^\alpha(\Theta)\}$. The claim is proved.
\end{proof}

\begin{cor}
Let $D$ be an integral domain and let $R$ be a pseudo-polynomial retract $D$-algebra. Let $\Theta$ be a sharp pre-Jaffard family of $D$, and suppose that $\locpic_u(RT,T)$ is a free group for each $T\in\Theta$. Then, there is an exact sequence
\begin{equation*}
\locpic_u(R,D)\simeq\bigoplus_{T\in\Theta}\locpic_u(RT,T)
\end{equation*}
\end{cor}
\begin{proof}
If $\Theta$ is sharp, by definition there is an ordinal $\alpha$ such that $T_\alpha=K$. By the previous theorem, the cokernel of $\bigoplus_{T\in\Theta}\locpic_u(RT,T)\longrightarrow\locpic_u(R,D)$ is $\locpic_u(RK,K)$. By Corollary \ref{cor:Picu-K}, the latter is a quotient of the trivial group $\Pic_u(RK,K)$, and thus it is itself trivial. The claim is proved.
\end{proof}

\bibliographystyle{plain}
\bibliography{/bib/articoli,/bib/miei,/bib/libri}

\end{document}